\documentclass[12pt]{amsart}
\usepackage{amsmath,amssymb,amsthm,mathtools}
\usepackage{mathrsfs,bm}
\usepackage{thmtools,thm-restate}
\usepackage{tikz-cd}
\usepackage{tikz}
\usepackage{enumitem}
\usepackage{graphicx}
\usetikzlibrary{matrix,arrows.meta,bending}
\usepackage{color}
\usepackage[numbers]{natbib}
\usepackage{booktabs}
\usepackage{longtable}
\usepackage{pgfplots}
\pgfplotsset{compat=1.14} 
\usetikzlibrary{arrows.meta,decorations.pathreplacing,positioning,calc} 

\usepackage[a4paper, left=3cm, right=3cm, top=3cm, bottom=3cm, footskip=15mm]{geometry}

\newtheorem{theorem}{Theorem}[section]

\newtheorem{proposition}[theorem]{Proposition}
\newtheorem{corollary}[theorem]{Corollary}
\newtheorem{conjecture}[theorem]{Conjecture}
\theoremstyle{definition}
\newtheorem{definition}[theorem]{Definition}
\newtheorem{example}[theorem]{Example}

\theoremstyle{remark}
\newtheorem{remark}[theorem]{Remark}

\numberwithin{equation}{section}

\makeindex

\usepackage{fancyhdr}
\usepackage{calc} 

\AtBeginDocument{%
  \setlength{\textwidth}{\dimexpr\paperwidth - 6cm\relax}%
  \setlength{\oddsidemargin}{\dimexpr 3cm - 1in\relax}%
  \setlength{\evensidemargin}{\dimexpr 3cm - 1in\relax}%
  \setlength{\marginparwidth}{2.5cm}%
}

\begin{document}

\title{Balanced matrices}

\author{T. Agama and G. Kibiti}
\address{Department of Mathematics, African Institute for Mathematical science, Ghana
}
\email{theophilus@aims.edu.gh/emperordagama@yahoo.com}

\address{Department of Mathematics, African Institute for Mathematical science, Ghana
}
\email{Gael@aims.edu.gh}

\subjclass[2000]{Primary 54C40, 14E20; Secondary 46E25, 20C20}

\date{\today}

\dedicatory{}

\keywords{balanced, fully-balanced, vertically balanced, horizontally balanced, discrepancy}

\begin{abstract}
In this paper, we introduce a particular class of matrices. We study the concept of a matrix to be \emph{balanced}. We study some properties of this concept in the context of matrix operations. We examine the behaviour of various matrix statistics in this setting. The crux will be to understand the determinants and the eigenvalues of balanced matrices. It turns out that there exist a direct communication among the leading entry, the trace, determinants and, hence, the eigenvalues of these matrices of order $2\times 2$. These matrices have an interesting property that allows us to predict their quadratic forms using their spectrum, without an information about their entries.
\end{abstract}

\maketitle

\section{Introduction and motivation}

Matrix theory is a central pillar of modern linear algebra and its many applications, ranging from numerical analysis and optimization to data science and engineering. Classical references that develop the algebraic and spectral machinery we rely on include standard texts on matrix analysis and applied linear algebra \cite{meyer2000matrix,roman2005advanced}. These works provide the language of eigenvalues, quadratic forms, and matrix norms that will be used throughout this paper.\\

In this work we introduce and study a new structural class of real matrices which we call \emph{balanced matrices}. Informally, a fully-balanced square matrix is one whose rows (and likewise its columns) display approximately the same ``energy'' in the sense of comparable sums of squares of entries; see Definition \ref{balanced} for the precise formulation. Balanced matrices sit between completely unstructured matrices and highly symmetric examples (such as scalar multiples of the identity or the all-ones matrix), and they capture a simple but useful form of row-column regularity that has striking consequences for basic matrix statistics. The motivating observation is that when the row /column energies are approximately equal, several otherwise independent quantities - leading entries, row /column sums and differences, trace, determinant, and the spectral extrema - tend to correlate in predictable ways for low-dimensional examples. The goal of the paper is to make these relationships explicit, rigorous (where possible), and to explore their consequences for quadratic forms and determinant behaviour under basic matrix operations.\\

More concretely, the principal contributions of the paper are as follows.\\

\begin{enumerate}
\item We formalize the notions of \emph{horizontally balanced}, \emph{vertically balanced} and \emph{fully-balanced} matrices (Definition \ref{balanced}), and we collect elementary closure properties for fully-balanced matrices under transpose, scalar multiplication, addition and multiplication in the $2\times2$ setting (Theorem \ref{balanced1}). These results show that the balanced property is robust under many natural algebraic operations, making the class a reasonable object of further study.
\bigskip

\item For $2\times2$ fully-balanced matrices with nonnegative entries, we relate simple entrywise combinations (row/column sums and differences) to the spectral extrema (Theorem \ref{eigenproof}). Roughly speaking, sums of entries along rows or columns approximate the maximal eigenvalue, while absolute differences approximate the minimal eigenvalue. This gives inexpensive, entrywise diagnostics that predict spectral behaviour without solving the characteristic equation.
\bigskip

\item We investigate how common matrix statistics interact in the balanced setting. In particular, we establish approximate homomorphism behaviour for the determinant under addition when one matrix has a negligible minimal eigenvalue and the other has no large outliers (Theorem \ref{homomorphism}). This clarifies when the ordinarily-wild determinant can be expected to behave additively on structured inputs.
\bigskip

\item We introduce a notion of \emph{discrepancy} (average row/column deviation) and use it to formalize propagation phenomena: fair discrepancy on one row (or column) forces fairness throughout a $2\times2$ balanced matrix (Theorem \ref{smaller system 1} and related remarks). We state natural higher-dimensional conjectures (Conjectures \ref{smaler systems} and \ref{edos}) that propose the existence of balanced interiors and discrepancy propagation in larger matrices.
\bigskip

\item Restricting to symmetric $2\times2$ balanced matrices, we show that the associated quadratic form can be written (up to small error) directly in terms of the spectral extrema (Proposition \ref{quadratic forms}). In particular, one may reconstruct an approximate quadratic form from the eigenvalues alone, an observation that may be useful in settings where only spectral information is available.
\end{enumerate}
\bigskip

Two guiding heuristics underlie the stated results. First, energy balance across rows and columns suppresses the possibility of single-entry dominance (outliers), which in turn forces the various row/column aggregates to be comparable; this comparability is the key mechanism that links elementary entrywise expressions to spectral quantities. Second, for $2\times2$ matrices the algebraic relations among the trace, the determinant and eigenvalues are particularly rigid; coupling this rigidity with the balance hypothesis produces sharp approximate identities that are false in general without additional structure. Throughout the paper, we exploit these low-dimensional algebraic identities as a laboratory for phenomena that we conjecture can be extended (with appropriate hypotheses) to higher dimensions.

\subsection*{Organization of the paper}
After fixing the notation and giving the formal definition of balanced matrices in Section~\ref{sec:def} (where simple examples are presented). In Section \ref{sec:elementary}, we prove elementary closure properties and collects basic transformations that preserve balance (Theorem \ref{balanced1}). In Section~\ref{sec:discrepancy}, we introduce the notion of discrepancy, prove propagation results for $2\times2$ matrices (Theorem \ref{smaller system 1}), and formulate conjectures intended to guide extensions to larger matrices. Section~\ref{sec:discrepancy} contains the result (Theorem \ref{homomorphism}) of the determinant additivity (approximate homomorphism) together with the hypotheses under which it is valid. In section \ref{sec:determinant}, we study the interaction between entries, trace, determinant, and spectrum for $2\times2$ fully-balanced matrices; the main spectral comparison result is Theorem \ref{eigenproof}. Section~\ref{sec:quadratic} treats quadratic forms associated with symmetric balanced matrices and proves Proposition \ref{quadratic forms} which expresses these forms in terms of spectral extrema. We conclude with a brief discussion of possible extensions, open problems, and directions for further research.

\subsection*{Notation and conventions}
We work over the real field unless otherwise noted. For a matrix $A=(a_{ij})$ we use $a_{i\cdot j}$ to denote the $(i,j)$-entry when that notation appears (consistent with the body of the paper), and $\mathcal{B}_n(\mathbb{R})$ denotes the class of $n\times n$ fully-balanced real matrices when such context is required. For background on the basic spectral identities used (trace-sum relation, determinant as product of eigenvalues, and quadratic form conventions), the reader may consult standard references \cite{meyer2000matrix,roman2005advanced}.
\bigskip

\section{Motivation}

Consider a typical $2 \times 2$ matrix of the form

\begin{align}
A := \begin{pmatrix} a & b \\ c & d \end{pmatrix}.\nonumber
\end{align}

A crucial step in understanding the behavior of this matrix is determining its spectrum, which consists of its eigenvalues. These eigenvalues can be found by solving the characteristic equation

\begin{align}
|A-\lambda I|=0,\nonumber
\end{align}

where $\lambda$ denotes any eigenvalue of $A$. The spectrum provides insight into the action of the matrix on a vector space, encapsulating information about scaling factors in linear transformations. However, calculating the spectrum, especially for higher-dimensional matrices, can be computationally intensive, necessitating sophisticated techniques.\\

For symmetric matrices, the quadratic form provides an alternative way to explore the various properties of the matrix. In the case of matrix $A$, the quadratic form is given by

\begin{align}
F(x,y):=ax^2+bxy+dy^2,\nonumber
\end{align}

which offers a geometric perspective on the influence of the matrix on the vectors in $\mathbb{R}^2$. The quadratic form is particularly important in optimization and geometry, where it aids in studying curvature and other key geometric properties.\\

Despite the wealth of tools available for matrix analysis, the computation of the spectrum and quadratic form typically requires a detailed understanding of the entries of a matrix. In what follows, we study a special class of matrices for which both the spectrum and the quadratic form can be effectively determined without the need to solve the characteristic equation or fully compute the matrix entries. In this class, the eigenvalues and quadratic forms are directly related to simple operations on the entries of the matrix, thus simplifying the analysis.\\

For any matrix $A$ in this class, we find that the sums of the row and column entries approximate the maximum eigenvalue in the spectrum:

\begin{align}
\sum \limits_{r=1}^{2} a_{i\cdot r} \approx \sum \limits_{s=1}^{2} a_{s\cdot j} \approx \max(M)\nonumber
\end{align}

and that the differences between the row and column entries approximate the minimum eigenvalue:

\begin{align}
|a_{i\cdot 1}-a_{i\cdot 2}|\approx |a_{1\cdot j}-a_{2\cdot j}| \approx \min(M)\nonumber
\end{align}

where $M$ is the spectrum of $A$. These simple relations provide a direct method for estimating the spectrum of the matrix, bypassing the need to solve complex characteristic equations.\\

Moreover, for symmetric matrices in this class, the quadratic form can be reconstructed directly from the eigenvalues, without requiring explicit knowledge of the matrix. In particular, the quadratic form is approximated by one of the following expressions, depending on the eigenvalues $\lambda_1$ and $\lambda_2$:

\begin{align}
F(x,y)\approx \bigg(\frac{\lambda_2-|\lambda_1|}{2}\bigg)(x+y)^2+ 2|\lambda_1|xy\nonumber
\end{align}
or
\begin{align}
F(x,y)\approx \bigg(\frac{\lambda_2+|\lambda_1|}{2}\bigg)(x+y)^2- 2|\lambda_1|xy.\nonumber
\end{align}

This framework does not only simplifies the process of matrix analysis, but also highlights a new class of matrices where key characteristics such as the spectrum and quadratic form can be efficiently deduced from elementary operations on the entries. This novel approach offers potential applications in areas that require rapid or simplified matrix diagnostics, particularly in high-dimensional settings where traditional methods may be computationally prohibitive.

\section{Balanced matrices}\label{sec:def}

\begin{definition}\label{balanced}
Let $\mathbb{M}_{n\times m}(\mathbb{R})$ be the space of $n\times m$ matrices with real entries. We say $A=(a_{ij})\in \mathbb{M}_{n\times m}(\mathbb{R})$, a non-zero matrix, is said to be \emph{horizontally balanced} if 
\begin{align}
\sum \limits_{j=1}^{m}a_{r\cdot j}^2\approx \sum \limits_{j=1}^{m}a_{s\cdot j}^2,\nonumber
\end{align}
for $1\leq s<r\leq n$. Similarly, it is said to be \emph{vertically balanced} if 
\begin{align}
\sum \limits_{i=1}^{n}a_{i\cdot r}^2\approx \sum \limits_{i=1}^{n}a_{i\cdot s}^2,\nonumber
\end{align}
for $1\leq s<r\leq 1$. Any matrix $A$ is said to be \emph{fully balanced} if it is both vertically and horizontally balanced.
\end{definition}

\begin{example}
Perhaps a good straight-forward example of a fully balanced matrix is the identity matrix, since it abides by the above criterion. Another obvious example of a fully-balanced matrix is given by \begin{align}
\lambda 
\begin{pmatrix}
1 & 1 &\cdots 1\\1 & 1 & \cdots 1\\ \vdots \\ 1 & 1 & \cdots 1 \end{pmatrix}\nonumber
\end{align}
for $\lambda \in \mathbb{R}$. Hence, for $A\in \mathbb{M}_{3}(\mathbb{R})$, the unity matrix
\begin{align}
A=
\begin{pmatrix}
1 & 1 & 1\\ 1 & 1 & 1\\ 1 & 1 & 1
\end{pmatrix}\nonumber            
\end{align}
the definition \ref{balanced} about fully balanced-matrices holds, for we have 
\begin{align}
1^2+1^2+1^2=1^2+1^2+1^2=3~\quad \mathrm{Horizontally}
\end{align}
\begin{align}
1^2+1^2+1^2=1^2+1^2+1^2=3 \quad \mathrm{Vertically}.
\end{align}
\end{example}
\bigskip
Through out this paper, we will choose for simplicity to specialize our study to fully-balanced square matrices. Letting the paper to be taken this way brings more questions around.

\section{Elementary properties of fully balanced matrices}\label{sec:elementary}

In this section, we examine some properties of fully balanced matrices. We investigate how these properties are preserved under various matrix operations. We prove results for the sums and products of $2\times 2$ matrices. Later, we will prove a result that will enable us to extend these properties to higher order matrices.

\begin{theorem}\label{balanced1}
Let $A, B\in \mathbb{M}_{n}(\mathbb{R})$ be fully-balanced matrices and let $\lambda\in \mathbb{R}$. The following hold:
\begin{enumerate}
\item [(i)] The transpose $A^{T}$ is also fully balanced.
\bigskip

\item [(ii)] The multiple $\lambda A$ is also fully balanced.
\bigskip

\item [(iii)] The sum of any $2\times 2$ fully-balanced matrix with positive entries is still fully-balanced. In other words, the notion of balanced balanced matrices is preserved under matrix addition.
\bigskip

\item [(iii)] The product of any $2\times 2$ fully-balanced matrices with positive entries is still fully balanced. 
\bigskip

\item [(iv)] The inverse of any $2\times 2$ non-singular fully-balanced matrix is still fully balanced. That is, if $A$ is a non-singular fully-balanced $2\times 2$ matrix, then so is $A^{-1}$.
\end{enumerate}
\end{theorem}

\begin{proof}
\begin{enumerate}
\item [(i)]Let $A=(a_{ij})\in \mathbb{M}_{n}(\mathbb{R})$ be fully-balanced balanced. By definition \ref{balanced}, it is both vertically and horizontally balanced. Since the transpose of a vertically balanced matrix becomes a horizontally-balanced matrix and vice-versa, it follows that the transpose $A^{T}$ must be fully balanced.
\bigskip

\item [(ii)] The fact that $\lambda A$ is also fully balanced is obvious.
\bigskip

\item [(iii)] Consider the $2\times 2$ matrices 
\begin{align}
A=
\begin{pmatrix}a_1 & b_1\\c_1 & d_1
\end{pmatrix}
\quad 
B=
\begin{pmatrix}
a_2 & b_2\\c_2 & d_2
\end{pmatrix}\nonumber      
\end{align}
By definition \ref{balanced} the following holds 
\begin{align}
a_1^2+b_1^2\approx c_1^2+d_1^2,\quad a_2^2+b_2^2\approx c_2^2+d_2^2 \label{sum 1}
\end{align}
and 
\begin{align}
b_1^2+d_1^2\approx a_1^2+c_1^2,\quad a_2^2+c_2^2\approx b_2^2+d_2^2.\label{sum 2}
\end{align}
Using the relation $a_1^2+b_1^2\approx c_1^2+d_1^2$ and $ a_1^2+c_1^2\approx b_1^2+d_1^2$, we observe that $c_1^2\approx b_1^2$. Since the entries are positive, we must have $c_1\approx b_1$. Using the equation further shows that $a_1\approx d_1$, $a_2\approx d_2$ and $b_2\approx c_2$. Their sum is given by \begin{align}
A+B=
\begin{pmatrix}
a_1+a_2 & b_1+b_2\\c_1+c_2 & d_1+d_2             \end{pmatrix}.\nonumber
\end{align}
We claim that the matrix $A+B$ is also fully balanced. We observe that 
\begin{align}
(a_1+a_2)^2+(b_1+b_2)^2&=a_1^2+a_2^2+2|a_1||a_2|+b_1^2+b_2^2+2|b_1||b_2|\nonumber \\&=(a_1^2+b_1^2+2|b_1||b_2|)+(a_2^2+b_2^2+2|a_1||a_2|)\nonumber\\& \approx (c_1^2+d_1^2+2|b_1||b_2|) + (c_2^2+d_2^2+2|a_1||a_2|)\\&\approx (c_1^2+c_2^2+2|c_1||c_2|)+(d_1^2+d_2^2+2|d_1||d_2|)\nonumber \\&\approx (c_1+c_2)^2+(d_1+d_2)^2\nonumber
\end{align}
by using the relations in \ref{sum 1} and \ref{sum 2}. Thus, the matrix $A+B$ is horizontally balanced. Similarly, we observe that \begin{align}
(b_1+b_2)^2+(d_1+d_2)^2&=b_1^2+b_2^2+2|b_1||b_2|+d_1^2+d_2^2+2|d_1||d_2|\nonumber \\&=(b_1^2+d_1^2+2|d_1||d_2|)+(d_1^2+d_2^2+2|b_1||b_2|)\nonumber \\&\approx (a_1^2+a_2^2+2|a_1||a_2|)+(c_1^2+c_2^2+2|c_1||c_2|)\nonumber \\&\approx (a_1+a_2)^2+(c_1+c_2)^2\nonumber
\end{align}
where we have used the relation \ref{sum 1} and \ref{sum 2}. Thus the matrix $A+B$ is also vertically balanced. Therefore, it must be fully balanced.
\bigskip

\item [(iv)] We now show that their product is also fully balanced. We get for their product
\begin{align}
AB&=
\begin{pmatrix}
a_1 & b_1\\c_1 &d_1
\end{pmatrix}
\begin{pmatrix}
a_2 & b_2\\c_2 & d_2
\end{pmatrix}\nonumber
\\&=
\begin{pmatrix}
a_1a_2+b_1c_2 & a_1b_2+b_1d_2\\c_1a_2+d_1c_2 & c_1b_2+d_1d_2
\end{pmatrix}.\nonumber
\end{align}
It follows that 
\begin{align}
(a_1a_2+b_1c_2)^2+(a_1b_2+b_1d_2)^2&=(a_1^2a_2^2+a_1^2b_2^2)+(b_1^2c_2^2+b_1^2d_2^2)+2a_1a_2b_1c_2+2a_1b_2b_1d_2\nonumber \\&\approx a_1^2(c_2^2+d_2^2)+b_1^2(c_2^2+d_2^2)+2d_1a_2c_1c_2+2b_2c_1d_1d_2\nonumber \\&\approx d_1^2(c_2^2+d_2^2)+c_1^2(a_2^2+b_2^2)+2d_1a_2c_1c_2+2b_2c_1d_1d_2\nonumber \\&\approx (c_1a_2+d_1c_2)^2+(c_1b_2+d_1d_2)^2\nonumber
\end{align}
and the product is horizontally balanced. A similar argument shows that, the product is also vertically balanced. Therefore, the product is fully balanced.
\bigskip

\item [(iv)] The fact that $A^{-1}$ is fully balanced, given that $A$ is fully balanced is obvious.
\end{enumerate}
\end{proof}

\section{Trace, determinants and eigenvalues associated with balanced matrices}\label{sec:determinant}

In this section, we examine various statistics associated with balanced matrices. We study the behaviour of their trace, their determinants, their eigenvalues, their eigenvectors and their corresponding interplay in this setting.

\begin{proposition}\label{entry}
Let 
\begin{align}
A=
\begin{pmatrix}
a & b\\c & d
\end{pmatrix}\nonumber
\end{align}
be a fully-balanced square matrix with positive real entries. If $a<\epsilon$, then $Tr(A)\leq N_{\epsilon}$ for any $\epsilon >0$ and where $N_{\epsilon}$ is a constant depending on $\epsilon$.
\end{proposition}

\begin{proof}
Using Theorem \ref{balanced1}, the result follows immediately.
\end{proof}

\begin{remark}
Theorem \ref{balanced1} relates the leading entry of a $2\times 2$ fully-balanced matrix to their trace. Indeed, if the leading entry is small enough then the trace must not be too big. Similarly, if the leading entry is somewhat large then their trace must be large. This property is archetypal of balanced matrices.
\end{remark}
\bigskip

Proposition \ref{entry} highlights the importance of balanced matrices. It suggests that the leading entry or more generally the diagonal entry of any $2\times 2$ fully-balanced matrices has a profound connection with their eigen-values, and hence influences their eigen-vectors. Indeed, by using the well-known elementary relation 
\begin{align}
\lambda_1+\lambda_2=Tr(A),\nonumber
\end{align}
where $\lambda_1$, $\lambda_2$ are the eigen-values of the fully-balanced matrix $A$, and using Proposition \ref{entry}, we observe that if the leading entry is small enough then each of the eigen-values must not be too big provided the spectrum is real and has only positive eigen values. Similarly, if the leading entry is somewhat large then at least one of the eigen-values must be large under the requirement of the structure of the spectrum. A similar description could be carried out to relate the leading entry of balanced matrices to their determinants, using the well-known elementary relation (see, e.g, \cite{roman2005advanced})
\begin{align}
det(A)=\lambda_1 \lambda_2,\nonumber
\end{align}
where each $\lambda_i$  for $1\leq i\leq 2$ is an eigenvalue of $A$. This is a description characteristic of very rare class of matrices of which balanced matrices is a sub-class.

\bigskip
Balanced matrices are very theoretically important and could have real use applications in areas of applied mathematics. The simple and the most basic example of a fully-balanced matrix, as we have seen, is the identity matrix. The determinant of this matrix is always $1$. This gives us a clue of the distribution of balanced matrices.

\begin{remark}
Henceforth, when we say a balanced matrix it will imply a fully-balanced matrix. Otherwise, we will specify the context of balanced matrix.
\end{remark}
\bigskip

Eigenvalues and eigenvectors are extremely important statistics in the study of matrices. Knowing these two for any matrix can be useful in practice. The quest to find an eigenvalue-value and, hence, eigenvector features very often in other various applied areas such as physics. The next result helps us to predict up to a smaller error eigenvalues and hence eigenvectors of balanced matrices, without the need to undergo the traditional procedure. This result relates the sums and differences of the entries of balanced matrices to the least and worst eigenvalue for $2\times 2$ balanced matrices. It will be useful to extend this result to matrices of higher orders. In the moment, we content ourselves with the following:

\begin{theorem}\label{eigenproof}
Let $A\in \mathbb{B}_{2}(\mathbb{R}^{+})$, the spaces of $2\times 2$ balanced matrices with each $a_{ij}\geq 1$. If $M=\{|\lambda_1|, |\lambda_{2}|\}$ is the set of  eigen-values of $A$, then
\begin{align}
\sum \limits_{r=1}^{2}a_{i\cdot r}\approx \sum \limits_{s=1}^{2}a_{s\cdot j}\approx \max (M)\nonumber
\end{align} 
for $1\leq s,r \leq 2$ and 
\begin{align}
|a_{i\cdot 1}-a_{i\cdot 2}|\approx |a_{1\cdot j}-a_{2\cdot j}|\approx \min(M)\nonumber
\end{align}
where $1\leq i,j\leq 2$.
\end{theorem}

\begin{proof}
Consider the fully-balanced matrix
\begin{align}
A=
\begin{pmatrix}
a & b\\c & d
\end{pmatrix}.\nonumber
\end{align}
Using the well-known elementary relation (See, e.g, \cite{roman2005advanced})
\begin{align}
det(B)=(-1)^{n}\lambda_1 \lambda_2\cdots \lambda_n\nonumber
\end{align}\begin{align}\lambda_1+\lambda_2+\cdots +\lambda_n=Tr(B)\nonumber
\end{align}
for any matrix $B$, we can write 
\begin{align}
det(A)&=\lambda_1\lambda_2\nonumber \\&=ad-bc\nonumber \\& \approx a^2-b^2 .\label{proof1}
\end{align}
Since $Tr(A)=\lambda_1+\lambda_2$, it follows that $|a-b|\approx |\lambda_1|$ and $a+b\approx |\lambda_2|$ or vice-versa. Similarly we can write 
\begin{align}
det(A)&=\lambda_1\lambda_2\nonumber \\&=ad-bc\nonumber \\& \approx d^2-c^2.\label{proof2}
\end{align}
Again, using the relation $Tr(A)=\lambda_1+\lambda_2$, it follows that $|c-d|\approx |\lambda_1|$ and $c+d\approx |\lambda_2|$ or vice-versa. Again it follows that 
\begin{align}
det(A)&=\lambda_1\lambda_2\nonumber \\&=ad-bc \nonumber \\& \approx a^2-c^2.\label{proof3}
\end{align}
Using the relation $Tr(A)=\lambda_1+\lambda_2$, it follows that $|a-c|\approx |\lambda_1|$ and $a+c\approx|\lambda_2|$ or vice-versa. Also, we have 
\begin{align}
det(A)&=\lambda_1 \lambda_2\nonumber \\&= ad-bc \nonumber \\& \approx d^2-b^2\label{proof4}
\end{align}
and it follows that $|b-d|\approx|\lambda_1|$ and $b+d\approx \lambda_2$ or vice-versa, by using the relation $Tr(A)=\lambda_1+\lambda_2$. Without loss of generality, we let $\lambda_{2}=\max(M)$ and $\min(M)=\lambda_1$. Then it follows that $b+d\approx a+c\approx a+b\approx c+d\approx \lambda_2=\max(M)$ and $|b-d|\approx |a-c|\approx |c-d|\approx |a-b|\approx \lambda_1=\min(M)$. If we suppose that $b+d\approx |b-d|$, then it follows that either $d\approx 0$ or $b\approx 0$, which contradicts the minimality of each of the $a_{ij}$'s. Similarly, let us suppose that $b+d\approx |a-c|$. It follows that
\begin{align}
b+d& \approx a-c\nonumber \\& \approx d-c\nonumber
\end{align}
and we have that $b\approx -c$ if and only if $c\approx 0$, which violates the minimality of $a_{ij}$ for $1\leq i,j\leq 2$. Also in the case where $b+d=-(a-c)$, then it follows that $d\approx 0$, which is a contradiction. Again if $b+d\approx |c-d|$, then we see that 
\begin{align}
b+d&\approx c-d\nonumber \\& \approx b-d\nonumber
\end{align}
and it follows that $d\approx 0$. On the other hand, we will have that $c\approx 0$, both of which contradicts the \textbf{minimality} of $a_{ij}$. Thus, using the fact that $A$ is fully-balanced, in a similar manner for other cases the result follows immediately.
\end{proof}

\begin{corollary}
Let $A_1, A_2\in \mathcal{B}_{2}(\mathbb{R}^{+})$ with $a_{ij}\geq 1$ and let $E_{\max}(A_1)$ denotes the maximum eigen-value of $A_1$. Then 
\begin{align}
E_{\max}(A_1+A_2)\approx E_{\max}(A_1)+E_{\max}(A_2).\nonumber
\end{align}
\end{corollary}

\begin{proof}
Consider the $2\times 2$ fully-balanced matrices given by \begin{align}
A_1:=
\begin{pmatrix}
a_1 & b_1\\c_1 & d_1
\end{pmatrix} 
\mathrm{and} \quad 
A_2:=
\begin{pmatrix}
a_2 & b_2\\c_2 & d_2
\end{pmatrix}.\nonumber
\end{align}
By Theorem \ref{balanced1}, their sum 
\begin{align}
A_1+A_2=
\begin{pmatrix}
a_1+a_2 & b_1+b_2\\c_1+c_2 & d_1+d_2
\end{pmatrix}\nonumber
\end{align}
is also fully-balanced. By Theorem \ref{eigenproof}, $E_{\max}(A_1+A_2)\approx b_1+b_2+d_1+d_2$, and the result follows immediately.
\end{proof}

\begin{remark}
Before we state the next result, we review the following terminologies concerning matrices in general.
\end{remark}

\begin{definition}
By a block $n\times m$ matrix, we mean any matrix of the form \begin{align}
A=
\begin{pmatrix}
C_{11} & C_{12} & \cdots & C_{1n}\\ C_{21} & C_{22} & \cdots & C_{2n}\\ \vdots \\C_{m1} & C_{m2} & \cdots & C_{mn}
\end{pmatrix},\nonumber
\end{align}
and where each $C_{ij}$ is a sub-matrix of $A$ for $1\leq i \leq m$ and $1\leq j\leq n$.
\end{definition}

\begin{definition}
Let $A\in \mathbb{M}_{m\times n}(\mathbb{R})$ given by 
\begin{align}
A:=
\begin{pmatrix}
a_{1\cdot 1} & a_{1\cdot 2} \cdots & a_{1\cdot n}\\a_{2\cdot 1} & a_{2\cdot 2}\cdots & a_{2\cdot n}\\ \vdots \\ a_{m\cdot 1} & a_{m\cdot 2} & \ddots a_{m\cdot n}
\end{pmatrix}.\nonumber
\end{align}
We say that a matrix $B$ is an interior of $A$ if it is a sub-matrix of $A$.
\end{definition}

\begin{conjecture}\label{smaler systems}
Let $A\in \mathcal{B}_{n}(\mathbb{R})$ denote the space of square balanced-matrices. There exists some interior of $A$ that is also balanced.
\end{conjecture}

\begin{remark}
By thinking of a matrix as a system, Conjecture \ref{smaler systems} roughly speaking conveys the notion that, if a bigger system is balanced, then there must be a sub-system that is also balanced.
\end{remark}

\section{Discrepancies of fully-balanced matrices}\label{sec:discrepancy}

In this section, with the goal of proving some weaker versions of Conjecture \ref{smaler systems}, we introduce the notion of discrepancy of fully-balanced matrices. It turns out that Conjecture \ref{smaler systems} is somewhat easier to attack in this setting.

\begin{definition}
Let $A\in \mathbb{M}_{n\times m}(\mathbb{R})$. By the \emph{discrepancy} of the matrix $A$ along rows, we mean the value \begin{align}
\sum \limits_{j=1}^{m}a_{ij}.\nonumber
\end{align}
Similarly, by the discrepancy along columns, we mean the value
\begin{align}
\sum \limits_{i=1}^{n}a_{ij}.\nonumber
\end{align}
\end{definition}

\begin{definition}
Let $A\in \mathbb{M}_{m\times n}(\mathbb{R})$ and let
\begin{align}
M_i=\frac{1}{m}\sum \limits_{j=1}^{m}a_{ij}.\nonumber
\end{align}
We say that the discrepancy is fair along rows if for each $1\leq i\leq m$ then $|M_i-a_{ij}|<\epsilon$ for all $1\leq j\leq m$ for a small $\epsilon>0$.
\end{definition}
The discrepancy is \emph{unfair} along rows if for some $a_{ij}$ $(j=1,2\ldots m)$, there exist some $n_0$ such that 
\begin{align}
|M-a_{ij}|>N\nonumber
\end{align}
for all $N\geq n_0$. 

\begin{remark}
Next we prove some few propositions concerning fully-balanced matrices, in the context of discrepancy.
\end{remark}

\begin{theorem}\label{smaller system 1}
Let $A\in \mathcal{B}_{2\times 2}(\mathbb{R^{+}})$, the space of fully-balanced $2\times 2$ matrices. Then $A$ has a fair discrepancy along rows if and only if it has a fair discrepancy along columns.  
\end{theorem}

\begin{proof}
Let $A\in \mathcal{B}_{2\times 2}(\mathbb{R}^{+})$ and suppose that $A$ has a fair discrepancy along rows. Then it follows that for each $1\leq i\leq 2$
\begin{align}
|M_{i}-a_{ij}|<\epsilon\nonumber
\end{align}
for small arbitrary $\epsilon>0$ and for all $1\leq j\leq 2$. This implies that $|M_1-a_{1j}|<\epsilon$ and hence $a_{11}\approx a_{12}$ and $|M_2-a_{2j}|<\epsilon$ for $\epsilon>0$ and hence $a_{21}\approx a_{22}$. Since $A$ is fully-balanced, it follows that $a_{11}\approx a_{22}$ and $a_{21}\approx a_{12}$. It follows that $A$ must have a fair discrepancy along columns. The converse, on the other hand, follows a similar approach.
\end{proof}
\bigskip

\begin{proposition}
Let $A\in \mathcal{B}_{2\times 2}(\mathbb{R}^{+})$. If $A$ has a fair discrepancy on exactly one row, then it must have fair discrepancy on rows.
\end{proposition}

\begin{proof}
Specify $A\in \mathcal{B}_{2\times 2}(\mathbb{R}^{+})$ given by \begin{align}
A:=
\begin{pmatrix}
a & b\\c & d
\end{pmatrix}.\nonumber
\end{align}
Suppose that $A$ has a fair discrepancy along exactly one row. Without loss of generality, let us assume the fair discrepancy occurs on the first row. By Theorem \ref{smaller system 1}, we have $a\approx b$. Since $A$ is fully-balanced, Theorem \ref{balanced1} gives $a\approx b\approx c\approx d$. This implies that $A$ has a fair discrepancy on rows. 
\end{proof}

\begin{conjecture}\label{edos}
Let $\epsilon>0$ and $A\in \mathbb{M}_{n\times m}(\mathbb{R})$ be a fully balanced matrix. The average discrepancy along rows is given by 
\begin{align}
M=\frac{1}{m}\sum \limits_{j=1}^{m}a_{ij}.\nonumber
\end{align} 
If 
\begin{align}
|M_i-a_{ij}|<\epsilon\nonumber
\end{align}
for a fixed $1\leq i\leq n$, then $|M_i-a_{ij}|<\epsilon$ for all $1\leq i\leq n$. 
\end{conjecture}

\begin{remark}
Conjecture \ref{edos} implies that if the discrepancy of a fully-balanced matrix along a given row is fair, then it must be fair on all other rows. In other words, a fair discrepancy on a given row is propagated to all other rows.
\end{remark}
\bigskip

In general the determinant of matrices is not an approximate homomorphism. That is, the determinant of the sums of matrices may not have the same distribution as the sum of each determinant. These two statistics may be close to each other and could very well be far from each other. Here is where the concept of balanced matrices plays an important role. Given $k$ distinct matrices, we say the determinant is an approximate homomorphism if the relation holds:
\begin{align}
det\bigg(\sum \limits_{k=1}^{n}A_k\bigg)\approx \sum \limits_{k=1}^{n}det(A_k).\nonumber
\end{align}
The next result clarifies and gives a more formal context to the ensuing discussion.

\begin{theorem}\label{homomorphism}
Let $A, B\in \mathcal{B}_{2}({\mathbb{R}^{+}})$, where $\mathcal{B}_{2}({\mathbb{R}^{+}})$ is the space of $2\times 2$ balanced-matrices with $a_{ij}\geq 1$ and $b_{ij}\geq 1$. Let $\mathcal{M}=\{|\lambda_1|,|\lambda_2|\}$ be the spectrum of $A$. If $\min(\mathcal{M})\approx 0$ and $B$ has a fair discrepancy along rows or columns, then 
\begin{align}
det(A+B)\approx det(A)+det(B).\nonumber 
\end{align}
\end{theorem}

\begin{proof}
Consider the $2\times 2$ fully-balanced matrices
\begin{align}
A=
\begin{pmatrix}
a_1 & a_2\\a_3 & a_4
\end{pmatrix} 
\quad \mathrm{and} \quad 
B=
\begin{pmatrix}
b_1 & b_2\\b_3 & b_4\end{pmatrix}.\nonumber
\end{align}
By Theorem \ref{balanced1}, we have the fully-balanced matrix \begin{align}
A+B:=
\begin{pmatrix}
a_1+b_1 & a_2+b_2\\a_3+b_3 & a_4+b_4\end{pmatrix}.\nonumber
\end{align}
It follows that 
\begin{align}
det(A+B)&=(a_1+b_1)(a_4+b_4)-(a_2+b_2)(a_3+b_3)\nonumber \\&=(a_1a_4-a_2a_3)+(b_1b_4-b_2b_3)+(a_1b_4+b_1a_4-a_2b_3-b_2a_3)\nonumber \\& \approx det(A)+det(B)+2(b_1a_1-a_2b_2)\nonumber \\&\approx det(A)+det(B)+2b_1(a_1-a_2)\nonumber
\end{align}
where we have utilized the fact that $A$ and $B$ are fully-balanced matrices, and that $B$ has a fair discrepancy along rows or columns. Using the fact that $\min(\mathcal{M})\approx 0$, the result follows from Theorem \ref{eigenproof}.
\end{proof}

\begin{remark}
Theorem \ref{homomorphism} tells us that the determinant can be made an approximate homomorphism on any two fully-balanced matrices of not-too-small entries, by making the least element in the spectrum of one matrix negligible and avoiding outliers in the entries and rows of the second.
\end{remark}

\begin{conjecture}
Let $A, B\in \mathcal{B}_{n}({\mathbb{R}^{+}})$, where $\mathcal{B}_{n}({\mathbb{R}^{+}})$ is the space of $n\times n$ balanced-matrices with $a_{ij}\geq 1$ and $b_{ij}\geq 1$. Let $\mathcal{M}=\{|\lambda_1|,|\lambda_2|,\ldots,|\lambda_n|\}$ be the spectrum of $A$. If $\min(\mathcal{M})\approx 0$ and $B$ has a fair discrepancy along rows or columns, then  
\begin{align}
det(A+B)\approx det(A)+det(B).\nonumber 
\end{align}
\end{conjecture}

\section{Quadratic forms associated with balanced matrices}\label{sec:quadratic}

In this section, we examine various forms associated with balanced matrices. For the time being we study the quadratic forms associated with fully-balanced $2\times 2$ matrices. We review therefore the following definitions.

\begin{definition}
Let 
\begin{align}
A:=
\begin{pmatrix}
a & b\\c & d
\end{pmatrix}\nonumber
\end{align}
be any symmetric matrix. By the quadratic form of $A$, we mean the expression
\begin{align}
F(x,y):=ax^2+2bxy+dy^2.\nonumber
\end{align}
\end{definition}
\bigskip

Let 
\begin{align}
A=
\begin{pmatrix}
a & b \\c & d 
\end{pmatrix}\nonumber
\end{align}
be a fully-balanced symmetric matrix. The associated quadratic form can be written as 
\begin{align}
F(x,y)&:=ax^2+2bxy+dy^2\nonumber \\& \approx a(x^2+y^2)+2bxy\nonumber \\& \approx a(x+y)^2+2(b-a)xy.\nonumber
\end{align}
Using Theorem \ref{eigenproof}, we can write 
\begin{align}
F(x,y)&:\approx a(x+y)^2+2(b-a)xy.\nonumber \\& \approx \bigg(\frac{\lambda_2-|\lambda_1|}{2}\bigg)(x+y)^2+2|\lambda_1| xy\nonumber
\end{align}
if $b>a$. Similarly if $b<a$, then the quadratic form looks a lot like
\begin{align}
F(x,y)&=ax^2+2bxy+dy^2\nonumber \\& \approx \bigg(\frac{\lambda_2+|\lambda_1|}{2}\bigg)(x+y)^2-2|\lambda_1| xy,\nonumber
\end{align}
where $\lambda_2$ and $\lambda_1$ are the worst and the least eigenvalues of $A$ respectively. More formally we launch a proposition concerning the quadratic forms associated with $2\times 2$ fully-balanced symmetric matrices. 

\begin{proposition}\label{quadratic forms}
Let $A$ be a fully-balanced $2\times 2$ symmetric matrix, such that $a_{ij}\geq 1$ for $1\leq i,j\leq 2$. Let $\mathcal{N}:=\{|\lambda_1|,|\lambda_2|\}$ be the spectrum of $A$, and where $\max(\mathcal{N})=|\lambda_2|$ and $\min(\mathcal{N})=|\lambda_1|$. Then one of the following is an approximation of the quadratic form of $A$ 
\begin{align}
F(x,y):\approx \bigg(\frac{\lambda_2-|\lambda_1|}{2}\bigg)(x+y)^2+2|\lambda_1|xy\nonumber
\end{align}
or 
\begin{align}
F(x,y):\approx \bigg(\frac{\lambda_2+|\lambda_1|}{2}\bigg)(x+y)^2-2|\lambda_1|xy.\nonumber
\end{align}
\end{proposition}

\begin{proof}
The result follows from the ensuing discussion concerning quadratic forms of fully-balanced matrices.
\end{proof}

\begin{remark}
Proposition \ref{quadratic forms} tells us that we do not necessarily need the entries of a $2\times 2$ symmetric fully-balanced matrices to compute the values of their quadratic forms. Given the eigenvalues of $A$, we can with some precision predict the quadratic form of any fully-balanced $2\times 2$ symmetric matrices, without knowing the entries. 
\end{remark}

\section{Further remarks}

In this paper, we have introduced the concept of balanced matrices, where we studied various matrix statistics underlying this concept. Much emphasis was placed on $2\times 2$ fully-balanced matrices. This can be considered as the beginning of a series of papers to study this concept. There is much optimistic work in progress to extend these results for lower order square matrices to matrices of higher orders. Another quest, in the not too distant future, will be to find if there really exist some bit of interaction between this class of matrices and matrices in general. This could provide a new window through which to study matrix theory.

\footnote{
\par
.}%

\bibliographystyle{amsplain}

\end{document}